\documentclass[12p]{amsart}
\usepackage[top=1.5in, bottom=1.5in, left=1.5in, right=1.5in]{geometry}
\usepackage{fourier} 
\usepackage{latexsym}
\usepackage{amssymb}
\usepackage{amsmath}
\usepackage{amscd}
\usepackage{graphicx}
\usepackage{float}
\usepackage{mathrsfs} 
\usepackage{enumerate}
\usepackage{tikz}

\tikzset{node distance=1.5cm, auto}

\newtheorem{theorem}{Theorem}[section]
\newtheorem{proposition}[theorem]{Proposition}
\newtheorem{lemma}[theorem]{Lemma}

\theoremstyle{definition}
\newtheorem{definition}[theorem]{Definition}
\newtheorem{question}[theorem]{Question}
\theoremstyle{remark}
\newtheorem{example}[theorem]{Example}
\newtheorem{remark}[theorem]{Remark}

\def\<{\langle}
\def\>{\rangle}
\def\RR{\mathbb{R}}

\def\NN{\mathbb{N}}
\def\CC{\mathbb{C}}
\def\ZZ{\mathbb{Z}}

\def\T{\mathbf{T}}

\def\I{\mathfrak{I}}
\def\A{\mathscr{A}}
\def\T{\mathscr{T}}
\def\L{\mathscr{L}}
\def\sgn{\operatorname{sgn}}
\def\1{\mathbb{1}}
\def\I{\mathbb{I}}

\title{On the multivariate Fujiwara bound for exponential sums}
\date{\today}
\author{Jens Forsg{\aa}rd}
\address{Department of Mathematics, Texas A\&M University, College Station, TX 77843}
\email{jensf@math.tamu.edu}

\begin{document}

\maketitle

\begin{abstract}
We prove the multivariate Fujiwara bound for exponential sums: for a $d$-variate exponential sum $f$
with scaling parameter $\mu$, if $x$ is contained in the amoeba $\A(f)$, then the distance from $x$ to the Archimedean tropical variety associated to $f$ is at most $d \sqrt{d}\, 2\log(2 + \sqrt{3})/ \mu$.
If $f$ is polynomial, then the bound can be improved to $d \log(2 + \sqrt{3})$.
\end{abstract}

\section{Introduction}
\label{sec:Intro}
It is a classical problem to find an upper bound on the norms of the roots of a complex univariate polynomial 
$g(w) = \sum_{k=0}^n c_k w^{k}$
in terms of the coefficients $c$.
In applications the near optimal Fujiwara bound
 \begin{equation}
 \label{eqn:Fujiwara1}
 2 \max \left( \left|\frac{c_{n-1}}{c_n}\right|, \left|\frac{c_{n-2}}{c_n}\right|^{\frac12},
 \dots, \left|\frac{c_{1}}{c_n}\right|^{\frac1{n-1}}, \left|\frac{c_{0}}{2c_n}\right|^{\frac1n}\right)
 \end{equation}
 is in common use.
 Following the presentation of Carmichael \cite{Car18}, Fujiwara's theorem from \cite{Fuj16}
 can be stated as follows\footnote{Unfortunately, Fujiwara's
 paper has not been available to the author.}:
 no root of the algebraic equation $g(w) = 0$ can be of
 absolute value greater than the unique positive root of the equation
 \begin{equation}
 \label{eqn:Fujiwara2}
 |c_n| \sigma^{n} = \sum_{k=0}^{n-1} |c_k| \sigma^{k}.
 \end{equation}
Fujiwara's theorem is a consequence of the triangle inequality,
and has appeared on several instances in the modern literature. 
Deducing the bound \eqref{eqn:Fujiwara1} from Fujiwara's theorem is straightforward:
if $\sigma$ is strictly greater than 
each term of \eqref{eqn:Fujiwara1}, then
\[
\sum_{k=0}^{n-1} \left|\frac{c_k}{c_n} \right| \sigma^{k-n} 
<2^{1-n}+\sum_{k=1}^{n-1} 2^{k-n}= 1,
\]
from which the result follows.

It is the purpose of this note to deduce the multivariate Fujiwara bound.
The reader might feel justly perplexed as an algebraic hypersurface in $\CC^d$, for $d \geq 2$,
is non-compact; hence, it cannot be bounded in norm. The bound we present is 
of a different flavor. In \eqref{eqn:Fujiwara1}, the arguments of the maximum function 
are the roots of a certain tropical polynomial constructed from $g(w)$,
and Fujiwara's theorem bounds the distance from the amoeba of $g(w)$
to this tropical variety. 
The bound is manifested in the factor of two appearing in front of the
maximum operator in \eqref{eqn:Fujiwara1}.

We have chosen to take the approach of exponential sums.
That is, we consider a $d$-variate exponential sum
\begin{equation}
\label{eqn:f}
 f(z) = \sum_{k=0}^n c_k \,e^{\<\lambda_k,z\>},
\end{equation}
where $c_k \in \CC$ and $\lambda_k \in \RR^d$ for $k= 0, \dots, n$.
We collect the exponents $\lambda_k$ for $k=0, \dots, n$, in a \emph{support set} $\Lambda$. 
We will say that $f$ is \emph{polynomial} if $\Lambda \subset \ZZ^d$. 
The name reflects the simple fact that $f$ is polynomial if and only if there exists 
a polynomial $g$ such that $f(z) = g(e^z)$.

Let $z = x + i y$, where $x, y \in \RR^d$. 
The \emph{amoeba} $\A(f)$ of the exponential sum $f$ is defined as the projection of 
the zero locus $V(f) \in \CC^d$ under taking component-wise real parts, i.e., $\A(f) = \Re(V(f))$.
In the case of a polynomial exponential sum the amoeba $\A(f)$ coincides 
with the classical amoeba of the polynomial $g$,
introduced in \cite{GKZ94} as the image of $V(g)\subset (\CC^*)^d$ under the logarithmic absolute value map.

The ``tropicalization'' of $f$ is by definition, in this context, the ``tropical exponential sum'' 
\begin{equation}
\label{eqn:tropf}
 \hat f (x)= \max_{k=0, \dots, n}\left(\log|c_k| + \<\lambda_k, x\>\right).
\end{equation}
The tropical variety $V(\hat f)$, denoted $\T(f)$, is known as the Archimedean tropical variety of $f$.
For a point $x \in \RR^d \setminus \T(f)$, let $\iota(x)$ denote the (unique) index $k$ such that
$\hat f(x) = \log|c_k| + \<\lambda_k, x\>$. 

Define $\mu(\Lambda)$ to be the minimal (Euclidean) distance between two points of $\Lambda$. 
We will call $\mu$ the scaling parameter of $f$.
As the standard scalar product is a bilinear form,
we can simultaneously 
dilate the amoeba $\A(f)$ and the tropical variety $\T(f)$, at the cost of an adjungate 
dilation of the support set $\Lambda$. It follows that any bound on the (Euclidean) distance 
between $\A(f)$ and $\T(f)$ must take the parameter $\mu$ into account.

In this modern language, Fujiwara's bound is part $b)$ of the following theorem.
\begin{theorem}[Erg{\"u}r, Paouris, and Rojas, \cite{EPR14}]
\label{thm:EPR1dim}
Let $d=1$, let $\lambda_0 < \dots < \lambda_n$, and let $x \in \RR$.
\begin{enumerate}[a)]
\item If\/ $x$ is of distance at least\/ $\log(3)/\mu$ from
$\T(f)$, then $x \in \RR\setminus \A(f)$. 
\item If\/ $\iota = \iota(x)$
is $0$ or $n$, and $x$ is of distance at least\/ $\log(2)/\mu$ from
$\T(f)$, then $x \in \RR \setminus \A(f)$.
\end{enumerate}
\end{theorem}
We will give the multivariate Fujiwara bound in two versions, one for polynomial exponential
sums and one for general exponential sums.
This distinction is somewhat surprising. 
In Theorem~\ref{thm:EPR1dim}, as well as in the bounds given in 
\cite{EPR14} (to be compared with \cite{AKNR13}),
the generalization from polynomial to arbitrary exponential sums
is manifested in a division by the parameter $\mu$;
in the multivariate setting there are further complications.

\begin{theorem}
\label{thm:PolynomialDistanceBound}
Let $f$ be a polynomial exponential sum. If\/ $x \in \RR^d$ is of (Euclidean) distance at least 
\begin{equation}
\label{eqn:DegreeBound}
\delta \geq d\, \log\big(2+\sqrt{3}\big)
\end{equation}
from $\T(f)$, then $x \in \RR^d\setminus \A(f)$.
\end{theorem}

\begin{theorem}
\label{thm:DistanceBound}
Let $f$ be an exponential sum with $\mu = \mu(\Lambda)$. 
If $x \in \RR^d$ is of (Euclidean) distance at least 
\begin{equation}
\label{eqn:GeneralDegreeBound}
\delta \geq \frac{d\sqrt{d}}{\mu}\, 2\,\log\big(2+\sqrt{3}\big)
\end{equation}
from $\T(f)$, then $x \in \RR^d\setminus \A(f)$.
\end{theorem}

The bounds from Theorems \ref{thm:PolynomialDistanceBound} and \ref{thm:DistanceBound}
are not sharp, cf.\ Theorem \ref{thm:EPR1dim}.
Though, the above bounds implies the existence of sharp Fujiwara bounds,
which we denote $\Delta_d$ for polynomial exponential sums and $\hat \Delta_d(\mu)$
for general exponential sums.
In the case of polynomial exponential sums, the sharp bound $\Delta_d$
can be determined implicitly
as the unique positive zero of an explicit exponential series, see \S \ref{sec:ImplicitBounds}.
Such an implicit description of the sharp bound $\hat \Delta_d(\mu)$
is not know; see the discussion in \S \ref{ssec:Honeycomb}. 
It is clear that $\Delta _d \leq \hat \Delta_d (1)$, since polynomial exponential sums
is a special case of exponential sums. We will remark in \S \ref{ssec:Honeycomb} 
that this inequality is strict. Presumably, the extra factor of $\sqrt{d}$ appearing in \eqref{eqn:GeneralDegreeBound} is necessary.

In Theorem \ref{thm:EPR1dim} special attention is placed on the case when $\lambda$ is
a vertex of the Newton polytope of $f$. A similar analysis can be made in the general
case, however the bounds obtained are not significant improvements of \eqref{eqn:DegreeBound}
respectively \eqref{eqn:GeneralDegreeBound}, see Remark \ref{rem:VerticesBound}.

This note could be taken as a remark to the papers \cite{AKNR13} and \cite{EPR14},
where similar bounds were given in terms of the number of monomials $n$.
That \emph{fewnomial bound}, for general exponential sum, is of the form $\log(n)/\mu$.
Hence, our \emph{degree bound} implies that the fewnomial bound is sharp only if we allow for $d$
to be arbitrarily big. 
Consider, e.g., the case $d=2$.  For polynomials the sharp degree bound
(see Table~\ref{tab:dimension2}) is an improvement of the fewnomial bound
if $f$ has least eight terms.
The sharper bounds obtained through the methods presented in \S 2 yields,
still in the case $d=2$, an improvement of the fewnomial bound if $f$
has at least five terms.

\section{Implicit bounds determined by $\Lambda$}
\label{sec:ImplicitBounds}
In this section we will derive an implicit bound $\Delta(\Lambda)$ on the distance $\delta$ between a point in the amoeba $\A(f)$ and the Archimedean tropical variety $\T(f)$. We will work with a fix support set $\Lambda$. Our bound will be highly implicit. We will associate to $\Lambda$ a (finite) family of real univariate exponential sums, each with a unique positive root. The bound in question will be the maximum of these roots. It will be the task of later sections to make the bound explicit in the case of specific families of support sets $\Lambda$.

\begin{lemma}
\label{lem:DistanceInequality}
Assume that $x \in \RR^d$ is such that $x\in \RR^d\setminus \T(f)$
with $\iota = \iota(x)$.
Then, $x$ is of (Euclidean) distance at least $\delta > 0$
from $\T(f)$ if and only if for each $k$ we have that
\begin{equation}
\label{eqn:DistanceInequality}
|c_k| e^{\<\lambda_k, x\>} \leq |c_{\iota}| e^{\<\lambda_{\iota}, x\> - \delta |\lambda_k - \lambda_{\iota}|}.
\end{equation}
\end{lemma}

\begin{proof}
The closest point $y$ to $x$ on the line defined by the equation
\[
|c_k| e^{\<\lambda_k, y\>} =  |c_{\iota}| e^{\<\lambda_{\iota}, y\>}
\]
is of the form $y = x + \xi (\lambda_k - \lambda_\iota)$. It follows that 
\[
|c_k| e^{\<\lambda_k, x\>} =  |c_{\iota}| e^{\<\lambda_{\iota}-\lambda_k, y\>+ \<\lambda_k, x\>} 
= |c_{\iota}| e^{ \<\lambda_{\iota}, x\>-\xi|\lambda_k - \lambda_\iota|^2}.
\]
Hence, the lemma follows from the observation that the inequality
\[
|c_{\iota}| e^{ \<\lambda_{\iota}, x\>-\xi|\lambda_k - \lambda_\iota|^2} 
\leq |c_{\iota}| e^{ \<\lambda_{\iota}, x\>-\delta|\lambda_k - \lambda_\iota|}.
\]
is equivalent to the inequality $|\xi(\lambda_k - \lambda_\iota)| \geq \delta$.
\end{proof}

\begin{definition}
For a fix support set $\Lambda$ and index $\iota$, we define the \emph{characteristic 
exponential sum} $\Xi_\iota(\delta)$ by
\[
\Xi_\iota(\delta) = \sum_{k\neq \iota} e^{-\delta |\lambda_k-\lambda_\iota|}.
\]
\end{definition} 

Notice that even for $\Lambda \subset \ZZ^d$ the characteristic function is in general
not polynomial. For this reason the natural setting for the problem under consideration
is that of exponential sums.

\begin{theorem}
\label{thm:DistanceImpliesComplement}
Assume that $x \in \RR^d$ is of (Euclidean) distance at least $\delta > 0$
from $\T(f)$. Let $\iota = \iota(x)$. If, in addition, $\Xi_\iota(\delta) < 1$,
then $x \in \RR^d\setminus \A(f)$.
\end{theorem}

\begin{proof}
Using Lemma \ref{lem:DistanceInequality} we have that
\[
|f(z)| \geq  |c_{\iota}|\,e^{\< \lambda_\iota , x\>} - \sum_{k\neq \iota} |c_k| \,e^{\<\lambda_k,x\>}
\geq  |c_{\iota}|\,e^{\< \lambda_\iota , x\>}\left( 1 - \Xi_\iota(\delta)\right) >  0
\]
for any $z$ with $\Re(z) = x$.
\end{proof}

\begin{remark}
The function $\Xi_\iota$ is a strictly decreasing function for $\delta \in (0, \infty)$ such that, firstly, $\Xi_\iota(0) = n$ and, secondly, $\Xi_\iota(\delta) \rightarrow 0$ as $\delta \rightarrow \infty$. It follows that
the equation $\Xi_\iota(\delta) = 1$ has a unique positive solution which we denote by $\delta_\iota$.
The assumption
that $\Xi_\iota(\delta) < 1$ in Theorem \ref{thm:DistanceImpliesComplement} holds for any
$\delta > \delta_\iota$.
\end{remark}

\begin{definition}
Let $\Lambda$ be a support set. We define the \emph{$\Lambda$-distance bound}
$\Delta(\Lambda)$ to be the positive real number $\max(\delta_0, \dots, \delta_n)$.
\end{definition}

\begin{theorem}
Let $f$ be an exponential sum with support set $\Lambda$.
Assume that $x \in \RR^d$ is of (Euclidean) distance at least $\delta > \Delta(\Lambda)$
from $\T(f)$. Then, $x \in \RR^d\setminus \A(f)$.
\end{theorem}

\begin{proof}
This follows from the definition of $\Delta(\Lambda)$ and Theorem \ref{thm:DistanceImpliesComplement}.
\end{proof}

\section{Lopsidedness}

We will, in this section, work towards a converse result of Theorem \ref{thm:DistanceImpliesComplement}, which will be crucial to investigate the sharpness of 
the bounds from Theorems \ref{thm:PolynomialDistanceBound} and \ref{thm:DistanceBound}.

\begin{definition}
\label{def:lopsidedness}
The exponential sum \eqref{eqn:f} is said to be \emph{lopsided} at $x\in \RR^d$
with respect to $\iota$ if
\[
|c_{\iota}|\,e^{\< \lambda_\iota , x\>} > \sum_{k\neq \iota} |c_k| \,e^{\<\lambda_k,x\>}.
\]
Further more, $f$ is said to be lopsided at $x$ if there exists an index $\iota$ such that $f$ is
lopsided at $x$ with respect to $\iota$. The set of points $x$ such that $f$ is \emph{not} lopsided 
is called the \emph{lopsided amoeba} of $f$ and is denoted $\L(f)$.
\end{definition}

Upon examination of the proof of Theorem \ref{thm:DistanceImpliesComplement} the crucial property
is that $f(z)$ is lopsided at $x$ with respect to $\iota$. The inclusion $\A(f) \subset \L(f)$ follows from, e.g., \cite{For15}. Hence, lopsidedness implies that $x\in \RR^d \setminus \A(f)$. The lopsided amoeba admits several equivalent definitions, the most useful for our purposes is given in the following proposition.

\begin{proposition}
\label{pro:lopsidedness}
Let $\theta \in (S^1)^{n+1}$, and let
$f_\theta(z) = \sum_{k=0}^n c_k e^{i\theta_k}\,e^{\<\lambda_k,z\>}$.
Then, 
\[
\L(f) = \bigcup_{\theta}\,\A(f_\theta).
\]
\end{proposition}

\begin{proof}
By definition, the exponential sum is lopsided at $x$ if and only if there is no convex polygon with
side lengths $|c_0|e^{\<\lambda_0, x\>}, \dots, |c_n|e^{\<\lambda_n, x\>}$.
On the other hand, $x \in \A(f_\theta)$ if and only if there exists $z\in \CC$ with $\Re(z) = x$ such that
the complex numbers 
$c_0 e^{i\theta_0}\,e^{\<\lambda_0,z\>}, \dots, c_0 e^{i\theta_0}\,e^{\<\lambda_0,z\>}$,
viewed as vectors in $\RR^2$, forms the boundary of a convex polygon.
\end{proof}

We are concerned, in this paper, with approximations of the amoeba $\A(f)$ that only depends
on the norms of the coefficients $c_k$ for $k= 0, \dots, n$. The
lopsided amoeba is the universal such approximation. Indeed, 
Proposition \ref{pro:lopsidedness} states that if $x\in \L(f)$, then there exists a $\theta\in (S^1)^{n+1}$
such that $x \in \A(f_\theta)$. 

\begin{theorem}
\label{thm:DistanceConverse}
Let $\delta>0$ be such that $\Xi_\iota(\delta) \geq 1$. Then, there exists an exponential sum
$f$ with support $\Lambda$ and a point $x\in \RR^d$ such that, firstly, $x$ is of distance $\delta$ from
$\T(f)$, secondly, $\iota = \iota(x)$, and thirdly $x\in \A(f)$.
\end{theorem}

\begin{proof}
Fix an arbitrary point $x \in \RR^d$.
Set $|c_\iota| = 1$, and for each $k  \neq \iota$ define $|c_k|$ so that equality holds in \eqref{eqn:DistanceInequality}. By construction we have that $\iota(x) = \iota$. It follows from Lemma \ref{lem:DistanceInequality} that,
for any choice of arguments of the coefficients $c_k$ for $k= 0, \dots, n$, we have that
$x$ is of distance $\delta$ from $\T(f)$. Further more, by construction we have that
\[
\sum_{k\neq \iota} |c_k| \,e^{\<\lambda_k,x\>} = 
|c_{\iota}|\,e^{\< \lambda_\iota , x\>}\, \Xi_\iota(\delta) \geq  |c_{\iota}|\,e^{\< \lambda_\iota , x\>},
\]
and hence, for any choice of arguments of $c_k$ for $k=0, \dots, n$, it holds
that  $x\in \L(f)$. In particular, by Proposition \ref{pro:lopsidedness}, we can choose the
arguments of the coefficients $c_k$ for $k=0, \dots, n$ such that $x \in \A(f)$.
\end{proof}

\section{Polynomials}
The aim of this section is to explore upper bounds on $\Delta(\Lambda)$ when $\Lambda\subset\ZZ^d$. The proofs in this section are based on the following simple remarks related to the function $\Xi_\iota(\delta)$. Firstly, we have that $\lambda_k - \lambda_\iota \in \ZZ^d$. In particular, $|\lambda_k - \lambda_\iota|\geq 1$ if $k \neq \iota$. Secondly, the root $\delta_\iota$ increases if we act on $\Xi_\iota(\delta)$ by decreasing the magnitude $|\lambda_k - \lambda_\iota|$ for some $k$ (by shifting $\lambda_k$).

\begin{proof}[Proof of Theorem \ref{thm:PolynomialDistanceBound}]
Assume that $x\in \RR^d$ is of distance at least $\delta$ from $\T(f)$. Let $\iota = \iota(x)$, which is
well defined. We can assume without loss of generality that $\lambda_\iota = 0$. Thus,
\[
\Xi_0(\delta) = \sum_{k\neq \iota} e^{-\delta|\lambda_k|} < \sum_{\beta \neq 0} e^{-\delta|\beta|}
\]
where the rightmost sum is taken over all $\beta \in \ZZ^d$ such that $\beta \neq 0$.

In order to evaluate the latter sum we will subdivide $\ZZ^d \setminus 0$ into equivalent classes
defined by the component-wise sign function. There are ${d \choose m}$ ways of choosing
$m$ non-zero components of $\beta$, and there are $2^m$ ways of distributing signs among
those components. As we exclude the point $\beta = 0$ from the sum, we have that $m \geq 1$. 
For each choice as above we obtain a sum
\[
\sum_{\gamma\in \NN_+^m} e^{-\delta|\gamma|} <  \sum_{\gamma \in \NN_+^m} e^{-\frac{\delta}{m}\left(|\gamma_1|+ \dots + |\gamma_m|\right)},
\]
where the majorizing series is obtained by use of the inequality of arithmetic and geometric means.
All in all, with $\delta$ fulfilling \eqref{eqn:DegreeBound},
\[
\Xi_0(\delta)
< \sum_{m=1}^d{d \choose m}\,2^m\, \sum_{\gamma \in \NN_+^m} e^{-\frac{\delta}{m}\left(|\gamma_1|+ \dots + |\gamma_m|\right)}
 = e^{-\delta}\sum_{m=1}^d{d \choose m}\, \left(\frac{2}{1- e^{-\frac\delta m}}\right)^m
 < 1-\frac{1}{(2+\sqrt{3})^d}.
\]
Hence, by Theorem \ref{thm:DistanceImpliesComplement}, we have that $x \in \RR^d \setminus \A(f)$.
\end{proof}

\begin{remark}
As mentioned above, the bound \eqref{eqn:DegreeBound} of Theorem \ref{thm:PolynomialDistanceBound} is 
not sharp for any dimension $d$. When $d = 1$ it can be seen from Theorems \ref{thm:DistanceImpliesComplement}
and \ref{thm:DistanceConverse} that
\[
\Xi_0(\delta) < \sum_{k \in \ZZ^*} e^{-\delta|k|}  =\frac{2e^{-\delta}}{1-e^{-\delta}},
\]
which is less than or equal to one if and only if $\delta \geq \log(3)$.
For general $d$, to obtain a sharp bound, one should solve for $\delta$ the equation
\begin{equation}
\label{eqn:DegreeImplicitEquation}
\sum_{\beta \neq 0} e^{-\delta|\beta|}  = 1,
\end{equation}
where the sum is taken over all non-zero $\beta \in \ZZ^d$.
\end{remark}

\begin{remark}
\label{rem:VerticesBound}
The original Fujiwara bound is concerned with the norm of the largest root,
as in part $b)$ of Theorem~\ref{thm:EPR1dim}.
Similarly, the bounds from Theorems \ref{thm:PolynomialDistanceBound}
and \ref{thm:DistanceBound} can be sharpened if we consider only the case
when $\lambda_\iota$ is a vertex of the Newton polytope of $f$.
For such a vertex, we have that $\Lambda$ is strictly contained in a 
halfplane passing trough $\lambda_\iota$.
Therefor, it suffices to solve the equation given by replacing the right hand side of
\eqref{eqn:DegreeImplicitEquation} by two.
This does not yield a significant improvement. Applying the same method as in the proof of
Theorem~\ref{thm:PolynomialDistanceBound} we obtain the bound
\[
\delta \geq - d \, \log\left(\frac{3+\sqrt[d]{2}}{2} - \sqrt{\left(\frac{3+\sqrt[d]{2}}{2}\right)^2 - \sqrt[d]{2}}\right),
\]
where we note that the expression inside the logarithm tends to $1/(2+\sqrt{3})$ as $d$ tends to infinity.
\end{remark}

\subsection{The case $d=2$.}
Let us briefly discuss the case $d=2$. The bound from Theorem \ref{thm:PolynomialDistanceBound}
is $2\log\left(2+\sqrt{3}\right)$. This bound is far from sharp; an improved bound can be obtained from the proof
of that theorem using that
\[
|\beta| \geq \max(\beta_1, \beta_2).
\]
Which yields that $x \in \RR^2 \setminus \A(f)$ is assured as long as $x$ is of distance 
at least 
\[
\delta \geq \log\left(\frac{\sqrt{3}+\sqrt{2}}{\sqrt{3}-\sqrt{2}}\right)
\]
from $\T(f)$. 
Numerical computations, using the software \textsc{Mathematica},
suggests that the sharp bound is approximately $1.99508\dots$.
A comparison between these bounds can be found in Table~\ref{tab:dimension2}.

\begin{table}
\begin{center}
\renewcommand*\arraystretch{1.4}
\begin{tabular}{llcl|l}
\hline \hline
Theorem \ref{thm:PolynomialDistanceBound} 
	& $2\log\left(2+\sqrt{3}\right)$ & = & $2.63391\dots$ & 2.11239\dots\\
Improvement 
	& $\log\left(\frac{\sqrt{3}+\sqrt{2}}{\sqrt{3}-\sqrt{2}}\right)$ & = & $2.29243\dots$ &\\
Numerics
	& & & $1.99508\dots$ & 1.53538\dots\\
\hline \hline
\end{tabular}
\end{center}
\vskip10pt
\caption{A comparison of distance bounds for polynomial exponential sums in the case $d=2$.
To the right are sharpened bounds for the case when $\lambda$ is a vertex.}
\label{tab:dimension2}
\end{table}

\subsection{A lower bound on $\Delta_d$.}
\label{ssec:PolynomialLowerBound}
Recall that $\Delta_d$ denotes the sharp degree bound. 

\begin{proposition}
\label{pro:LowerPolynomialBound}
We have that $\Delta_d \geq \sqrt{d}\,\log(3)$.
\end{proposition}

\begin{proof}
Let $\delta <  \sqrt{d}\,\log(3)$, and let 
$\xi(\delta)$ be defined by the relation
\[
 \xi(\delta) = \left(3^d - 1\right) \frac{e^{-\sqrt{d}\,\delta}}{1 - e^{-\sqrt{d}\,\delta}} - 1,
\]
so that $\xi(\delta) > 0$.

Let $S$ denote the set of all nonzero vectors in $\ZZ^d$ with entries in $\{-1,0,1\}$.
Then, $\# S = 3^d - 1$.
Each such vector has norm at most $\sqrt{d}$, and they specify
$3^d -1$ distinct rays emerging from the origin.
Let $\Lambda_m$ consist of the union of the $m$ first integer points along each ray
defined by an element of $S$.
Then, for each $\varepsilon > 0$, there exists a $M = M(\varepsilon)$ such 
that for all $m \geq M$ it holds that
\[
\Xi_0(\delta) + \varepsilon 
\geq \sum_{s\in S} \frac{e^{-\sqrt{s}\,\delta}}{1 - e^{-\sqrt{s}\,\delta}}
 \geq \left(3^d - 1\right) \frac{e^{-\sqrt{d}\,\delta}}{1 - e^{-\sqrt{d}\,\delta}}.
\]
In particular, if we choose $\varepsilon < \xi(\delta)$, then
\[
\Xi_0(\delta) \geq 1+\xi(\delta) - \varepsilon > 1.
\]
In particular, it follows from Theorem \ref{thm:DistanceConverse} that $\Delta_d > \delta$.
Since $\delta < \sqrt{d}\,\log(3)$ was arbitrary, the result follows.
\end{proof}

\section{The general case}

Let us now consider the general case of $\Lambda \subset \RR^d$. 
Recall that $\mu = \mu(\Lambda)$ denotes the minimal (Euclidean) distance between
two points of $\Lambda$.
Our approach to prove Theorem \ref{thm:DistanceBound} is
the following: for each $\iota$, we will approximate $\Lambda$
with a subset of a dilation of $\ZZ^d$ such that the function $\Xi_\iota(\delta)$ can only increase.

\begin{proof}[Proof of Theorem \ref{thm:DistanceBound}]
Fix $\iota$. It suffices to show that $\Xi_\iota(\delta) < 1$ for $\delta$ fulfilling \eqref{eqn:GeneralDegreeBound}. Consider, for each $\lambda_k \in \Lambda$
with $k\neq \iota$ a closed $d$-dimensional ball $S_k$ centered at $\lambda_k$ with radius 
$\mu/2$. Since all pairs of points of $\Lambda$ are of distance at least $\mu$, two such spheres intersect in one point at the most.

Each sphere $S_k$ contains as a subset a $d$-dimensional cube $B_k$
with side lengths $l = {\mu}/{\sqrt{d}}$ and centre of mass at $\lambda_k$.
Consider the dilated lattice 
\[
\Gamma = \frac{\mu}{2\sqrt{d}}\,\ZZ^d.
\]
Then, either $\lambda_k \in \Gamma$ in which case we set $\hat \lambda_k = \lambda_k$,
or there exists a point $\hat \lambda_k \in \Gamma$ in the interior of $S_k$ such
that $|\hat \lambda_k| \leq |\lambda_k|$. Indeed, the vector $-\sgn(\lambda_k)$ 
(with arbitrary entries $\pm 1$ for the components of $\lambda_k$ which vanishes)
determines a ``quadrant'' of $B_k$ which is completely contained in the half space
$\{\xi\in \RR^d\, | \, |\xi| \leq |\lambda_k|\}$. Such a quadrant contains at least one point
of $\Gamma$ in its closure. If $\Gamma$ belongs to the boundary of $S_k$, then $\Gamma$ is
a vertex of $B_k$, in which case all vertices the corresponding quadrant of $B_k$,
including $\lambda_k$, belongs to $\Gamma$.

Set $\hat \lambda_\iota = \lambda_\iota$, and let $\hat \Lambda \subset \Gamma$ be the support set
containing all $\hat \lambda_k$ for $k=0, \dots, n$. Since $|\hat \lambda_k| \leq |\lambda_k|$,
it holds that
\[
\Xi_\iota(\delta) \leq \hat \Xi_\iota(\delta),
\]
hence it suffices to show that 
$\Xi_\iota(\delta)$ is smaller than one, for $\delta$ fulfilling \eqref{eqn:GeneralDegreeBound}.

Let $B \subset \ZZ^d$ be such that $\hat \Lambda=  \frac{\mu}{2\sqrt{d}}B$.
Let $\hat\lambda \in \hat \Lambda$, with $\hat\lambda=  \frac{\mu}{2\sqrt{d}}\beta$ for some
$\beta \in \ZZ^d$. It follows that $|\hat\lambda| =  \frac{\mu}{2\sqrt{d}}|\beta|$, and we can conclude 
that
\[
\hat \Xi_\iota(\delta) = \tilde \Xi_\iota\left(\frac{\mu}{2\sqrt{d}}\, \delta \right)
\]
where $\tilde\Xi_\iota$ is the characteristic function of $\beta_\iota$ with respect to the
support set $B$. It thus follows from Theorem \ref{thm:PolynomialDistanceBound} that
the right hand side is majorized by one if
\[
\frac{\mu}{2\sqrt{d}}\,\delta \geq d \log\left(2+\sqrt{3}\right),
\]
which is equivalent to \eqref{eqn:GeneralDegreeBound}.
\end{proof}

\subsection{On the relation between $\Delta_d$ and $\hat \Delta_d(\mu)$.}
\label{ssec:Honeycomb}
Recall that $\hat \Delta_d(\mu)$ denotes the sharp degree bound for exponential sums. 
Let us present a general construction, which serves as the basis of our examples.
Let $\1$ denote the $d\times d$-matrix all whose entries are equal to $1$, 
and let $\I$ denote the $d\times d$ identity matrix. 
Define a transformation $T$ by
\[
T = \frac{\1 \,\varepsilon + \I}{\sqrt{2}},
\]
where
\[
\varepsilon = \frac1d\left(\sqrt{1+d} - 1\right).
\]
Notice that the determinant of $T$ is ${\sqrt{1+d}}/{\sqrt{2}}$.
The matrix $T$ has as eigenvalues $\varepsilon/\sqrt{2}$, with an $(d-1)$-dimensional eigenspace
spanned by the vectors $e_j - e_k$ for $j\neq k$, and $\sqrt{1+d}/\sqrt{2}$, with eigenvector
$\1$. In particular, the spectral value of $T$ is $\sqrt{1+d}/\sqrt{2}$.
We will consider the lattice $T \ZZ^d$ generated by the columns of $T$,
where we notice in particular that $\mu(T\ZZ^d) = 1$.
Hence, $T$ alters areas and distances 
but leaves the parameter $\mu$ invariant when acting on the lattice $\ZZ^d$. 
One can show that for polynomials with few terms (i.e., at most $2(d+1)$ terms)
the lattice $T\ZZ^d$ is the ``worst possible'' in terms of maximizing the distance
from a point in $\A(f)$ to $\T(f)$.
However, there is no monotone relation
between $T\ZZ^d$ and arbitrary support sets in general. 

\begin{example}
In the case $d = 2$, the lattice $T \ZZ^d$ is the Honeycomb lattice. There are six
points in $T \ZZ^d$  of distance one from the origin, and six points of distance
$\sqrt{3}$. It follows, by an argument analogous to the proof of Proposition~\ref{pro:LowerPolynomialBound}, that $\hat \Delta_2(1)$ is strictly greater than the unique positive root
of the equation
\[
6 e^{-\delta} + 6 e^{-\sqrt{3}\, \delta} = 1.
\]
A numerical computation, aided by \textsc{Mathematica}, gives that this root is approximately $1.99984\dots$. Notice that $\mu = 1$ in this case. In particular, we have that
\[
\hat \Delta_2(1) \geq 1.99984\ldots > 1.99508\ldots = \mu\, \Delta_2,
\]
so that $\hat \Delta_2(1)$ is strictly greater than $\Delta_2$ in this case.
\end{example}

\begin{question}
\label{q:WorstCase}
Is $\sqrt{2}\,\hat \Delta_d(\mu) = \mu\,{\sqrt{1+d}}\,\Delta_d$?
\end{question}

\bibliographystyle{amsplain}

\renewcommand{\bibname}{References} 



\end{document}